\documentclass{amsart}

\usepackage{amssymb, enumerate}
\usepackage{amsrefs}

\theoremstyle{plain}
\newtheorem{theorem}{Theorem}
\newtheorem{lemma}[theorem]{Lemma}

\newtheorem{prop}[theorem]{Proposition}
\newtheorem{corollary}[theorem]{Corollary}

\theoremstyle{definition}
\newtheorem{definition}[theorem]{Definition}

\theoremstyle{remark}
\newtheorem{remark}[theorem]{Remark}

\newtheorem{claim}{Claim}

\newcommand{\lab}{\left|}
\newcommand{\rab}{\right|}
\newcommand{\bdry}{\partial_\infty}
\newcommand{\lp}{\left(}

\newcommand{\rp}{\right)}
\newcommand{\e}{\epsilon}
\newcommand{\R}{\mathbb{R}}

\newcommand{\T}{\mathbb{T}}

\newcommand{\Sp}{\mathbb{S}}
\newcommand{\D}{\mathbb{D}}

\newcommand{\Z}{\mathbb{Z}}

\newcommand{\N}{\mathbb{N}}
\newcommand{\cl}{\overline}
\newcommand{\ti}{\textit}

\DeclareMathOperator{\dist}{\textup{\text{dist}}}
\DeclareMathOperator{\diam}{\textup{\text{diam}}}
\DeclareMathOperator{\id}{\textup{\text{id}}}

\DeclareMathOperator{\Vol}{\textup{\text{Vol}}}

\DeclareMathOperator{\inter}{\textup{\text{int}}}

\DeclareMathOperator{\Ner}{\textup{\text{Ner}}}
\DeclareMathOperator{\conv}{\textup{\text{conv}}}
\DeclareMathOperator{\bc}{\textup{\text{bc}}}

\DeclareMathOperator{\const}{\textup{\text{const}}}
\DeclareMathOperator{\ALC}{\textup{\text{ALC}}}
\DeclareMathOperator{\LLC}{\textup{\text{LLC}}}

\numberwithin{equation}{section}
\numberwithin{theorem}{section}

\begin{document}

\title[Discrete length-volume inequalities]{Discrete length-volume inequalities and lower volume bounds in metric spaces}
\author{Kyle Kinneberg}

\date{\today}

\maketitle

\begin{abstract}
A theorem of W. Derrick ensures that the volume of any Riemannian cube $([0,1]^n,g)$ is bounded below by the product of the distances between opposite codimension-1 faces. In this paper, we establish a discrete analog of Derrick's inequality for weighted open covers of the cube $[0,1]^n$, which is motivated by a question about lower volume bounds in metric spaces. Our main theorem generalizes a previous result of the author in \cite{Kin14}, which gave a combinatorial version of Derrick's inequality and was used in the analysis of boundaries of hyperbolic groups. As an application, we answer a question of Y. Burago and V. Zalgaller about length-volume inequalities for pseudometrics on the unit cube.
\end{abstract}

\section{Introduction} \label{introsec}
There is a deep and well-studied relationship in metric geometry between the volume of a space and the lengths of curves that, in some way, generate it. An early example of such a relationship is due to K. Loewner (unpublished, but see \cite{Pu52} for a discussion) and deals with conformal structures on the torus $\T^2$.

\begin{theorem} \label{Loewner}
Let $(\T^2, g)$ be the 2-dimensional torus, equipped with a Riemannian metric $g$, let $\ell(g)$ denote the infimal length of a closed curve on $\T^2$ that is not homotopically trivial, and let $\Vol(g)$ denote the volume of $\T^2$ with respect to the metric $g$. Then $\Vol(g) \geq \tfrac{\sqrt{3}}{2}\ell(g)^2$,
and equality holds if and only if $(\T^2,g)$ is isometric to the flat torus $\R^2/\Lambda$, where $\Lambda$ is the lattice generated by $(1,0)$ and $(1/2,\sqrt{3}/2)$.
\end{theorem}

Loewner's inequality is only the beginning of a very rich body of work that has sought to understand similar phenomena for more general spaces and in dimensions greater than two. We refer to \cite[Chapter 4]{Grom99} for a broad survey of methods and results in this area. Of particular interest to us is the following theorem, originally proved by W. Derrick \cite[Theorem 3.4]{Der69}. Here we state it in the form cited in \cite{Grom99}.

\begin{theorem} \label{Derrick}
Let $([0,1]^n,g)$ be the $n$-dimensional unit cube, equipped with a Riemannian metric $g$. Let $F_k$ and $F_k'$ denote the pairs of opposite codimension-1 faces of $[0,1]^n$, for $1 \leq k \leq n$, and let $d_k$ be the distance between $F_k$ and $F_k'$ with respect to the metric determined by $g$. Then $\Vol(g) \geq d_1\cdots d_n$.
\end{theorem}

Let us outline the main ideas in the proof of Derrick's inequality; we will see them reappear in combinatorial form later in this paper. For each $k$, let $f_k \colon [0,1]^n \rightarrow \R$ be defined by $f_k(x) = \dist_{g}(x,F_k)$, where $\dist_{g}$ denotes the distance with respect to the metric determined by $g$. It is clear that $f_k(F_k) = \{0\}$ for each $k$, and the definition of $d_k$ ensures that $f_k(F_k') \subset [d_k,\infty)$. Moreover, each $f_k$ is 1-Lipschitz. For simplicity, let us assume that $f_k$ is smooth (in reality, one would approximate $f_k$ by smooth functions that are $(1+\e)$-Lipschitz), so the differential satisfies $||df_k || \leq 1$. Now define $f \colon [0,1]^n \rightarrow \R^n$ by $f(x) = (f_1(x),\ldots,f_n(x))$. If $\omega = dy_1 \wedge \cdots \wedge dy_n$ denotes the canonical volume form on $\R^n$, then its pull-back is $f^*\omega = df_1 \wedge \cdots \wedge df_n$. Hadamard's inequality then gives 
$$||f^*\omega|| \leq ||df_1|| \cdots ||df_n|| \leq 1.$$
Letting $\nu$ denote the Riemannian volume form on $([0,1]^n,g)$, we have
\begin{equation} \label{volg}
\Vol(g) \geq \int_{[0,1]^n} ||f^*\omega|| d\nu \geq  \lab \int_{[0,1]^n} f^*\omega \rab \geq \int_{f([0,1]^n)} \omega.
\end{equation}
As $f_k(F_k) = \{0\}$ and $f_k(F_k') \subset [d_k,\infty)$ for each $k$, standard topological arguments ensure that the rectangle $[0,d_1] \times \cdots \times [0,d_n]$ is contained in $f([0,1]^n)$ (see Lemma \ref{poly} below). Consequently, the product $d_1 \cdots d_n$ is a lower bound for the right-hand side of \eqref{volg}, and this gives the desired inequality.

In \cite[Section 4.3]{Kin14}, the present author proved a combinatorial version of Derrick's inequality for open covers of the unit cube $[0,1]^n$. This result was used in the construction of a metric with certain regularity properties on the boundary of a Gromov hyperbolic metric space. The set-up was as follows. Let $\mathcal{U} = \{U_i\}_{i \in I}$ be an open cover of $[0,1]^n$, and again let $F_k$ and $F_k'$ denote the pairs of opposite codimension-1 faces. We say that $U_{i_1},\ldots,U_{i_m}$ is a \ti{chain} if $U_{i_j} \cap U_{i_{j+1}} \neq \emptyset$ for each $j$. Moreover, such a chain is said to connect two sets $A$ and $B$ if $U_{i_1} \cap A \neq \emptyset$ and $U_{i_m} \cap B \neq \emptyset$.

\begin{theorem}[{\cite[Proposition 4.4]{Kin14}}] \label{combcube}
Let $d_k$ denote the smallest number of sets $U_i$ in a chain that connects $F_k$ and $F_k'$. Then $\#\mathcal{U} \geq d_1\cdots d_n$.
\end{theorem}

Note that although this result is analogous to Derrick's theorem, it does not parallel the Riemannian inequality. Indeed, the sets $U_i$ in $\mathcal{U}$ are essentially treated as if they all had diameter 1. The primary purpose of this paper is to extend the preceding result to a weighted version, which is closer to Theorem \ref{Derrick} and also generalizes Theorem \ref{combcube}. 

To this end, let $\mathcal{U} = \{U_i\}_{i \in I}$ be an open cover of $[0,1]^n$ as before, and let $w \colon I \rightarrow [0,\infty)$ be a corresponding weight function. One should think of $w(i)$ as the weight associated to the set $U_i$. Together, $\mathcal{U}$ and $w$ give us a discrete notion of distance on $[0,1]^n$. Namely, we define
$$\dist_w(A,B) = \inf \left\{ \sum_{j=1}^m w(i_j) : 
\begin{array}{l}
U_{i_1},\ldots,U_{i_m} \text{ is a chain} \\
\text{that connects } A \text{ and } B
\end{array} \right\}.$$
The sum $w(i_1) + \ldots + w(i_m)$ is said to be the length of the corresponding chain. By path connectedness of $[0,1]^n$ and compactness of paths, it is easy to see that any two points in the cube can be connected by a chain. In particular, $\dist_w(A,B)$ is finite for $A,B \subset [0,1]^n$. We should note that a chain might be disconnected topologically, as we have made no assumption on the connectedness of the sets in $\mathcal{U}$. Our main theorem in this paper is the following.

\begin{theorem} \label{singleweight}
Let $\mathcal{U}$ be an open cover of $[0,1]^n$, let $w$ be a corresponding weight function, and let $d_k = \dist_w(F_k,F_k')$ for each $1\leq k \leq n$. Then 
$$\sum_{i \in I} w(i)^n \geq d_1 \cdots d_n.$$
\end{theorem}

In fact, we will prove a more general version of this inequality that has more in common with the results of Derrick in \cite{Der71} and \cite{Der68}: we allow the discrete distance between $F_k$ and $F_k'$ to be taken with respect to (possibly) different weight functions for different values of $k$.

\begin{theorem} \label{LV}
Let $\mathcal{U}$ be an open cover of $[0,1]^n$, and let $w_k$ be associated weight functions for $1\leq k \leq n$. If $d_k = \dist_{w_k}(F_k,F_k')$ for each $k$, then
$$\sum_{i \in I} \lp \prod_{k=1}^n w_k(i) \rp \geq d_1 \cdots d_n.$$
\end{theorem}

It is clear that Theorem \ref{singleweight} follows immediately from Theorem \ref{LV} by setting $w_k = w$ for each $k$.

As a corollary to Theorem \ref{singleweight}, we will easily obtain lower \ti{Hausdorff content} bounds for continuous images of $[0,1]^n$ in arbitrary metric spaces. Recall that if $(X,d)$ is a metric space, the $Q$-dimensional Hausdorff content of a compact subset $E \subset X$ is defined to be
$$\mathcal{H}_Q^{\infty}(E) = \inf \left\{\sum_{i\in I} (\diam U_i)^Q : \{U_i\}_{i \in I} \text{ is an open cover of } E \right\}.$$
We will show the following bound.

\begin{corollary} \label{contentcor}
Let $g \colon [0,1]^n \rightarrow X$ be continuous. Then
$$\mathcal{H}_n^{\infty}(g([0,1]^n)) \geq \prod_{k=1}^n \dist(g(F_k),g(F_k')),$$
where $\dist$ denotes the metric distance between sets in $X$.
\end{corollary}

For example, if $([0,1]^n,d)$ is the unit cube equipped with an arbitrary metric whose topology coincides with the Euclidean topology, then one can apply this inequality to the identity function. We remark here that our definition of Hausdorff content does not include the normalizing multiplicative factor $\Vol(B^n)2^{-n}$, where $B^n$ is the unit ball in $\R^n$ and $\Vol(\cdot)$ is $n$-dimensional Lebesgue measure, as is standard in geometric measure theory and which appears in the results of Derrick. In this sense, Corollary \ref{contentcor} does not recover Derrick's inequality. On the other hand, Corollary \ref{contentcor} is sharp in the metric category, and it can also be used to answer a question posed by Y. Burago and V. Zalgaller in \cite[p.\ 296]{BZ88} concerning pseudometrics on $[0,1]^n$. This is discussed further in Section \ref{metric}.

More generally, Corollary \ref{contentcor} provides a method of verifying lower volume bounds in fairly general classes of metric spaces. Namely, if one can find continuous images of cubes in a metric space $X$ whose sides are well-separated, then $X$ must satisfy some corresponding lower volume bounds. This leads us to study metric spaces that ``admit fat cubes." It turns out that some standard types of connectivity properties (that are common, for example, in the study of boundaries of hyperbolic groups) ensure that the metric space admits fat squares.

The outline of the paper is as follows. In Section \ref{prelim}, we will introduce some techniques that appear frequently: partitions of unity, the nerve of open covers, and the topological non-degeneracy lemma that we used in the above discussion on Derrick's inequality. Section \ref{cube} will be devoted entirely to the proof of Theorem \ref{LV}. In Section \ref{metric}, we take up the topic of lower Hausdorff content bounds in metric spaces, and the final section extends these considerations to the study of metric surfaces.

\subsection*{Acknowledgments}
This paper is part of the author's Ph.D. thesis at UCLA. He thanks his advisor, Mario Bonk, for many helpful discussions and recommendations about this project. He also thanks Andrey Mishchenko for bringing the (crucial) paper \cite{Sch93} to his attention and John Mackay for helpful feedback.

The author gratefully acknowledges partial support from NSF grant DMS-1162471 during his work on this project.

\section{Notation and preliminaries} \label{prelim}

Let $(X,d)$ be a metric space. As is standard, we will use $B(x,r)$ to denote the open ball centered at $x \in X$ with radius $r >0$. For subsets $A,B \subset X$, we let
$$\dist(A,B) = \inf \{d(x,y): x \in A \text{ and } y \in B\}$$
be the distance between $A$ and $B$. In the case that $A = \{x\}$ we abuse notation and simply write $\dist(x,B)$. We also let
$$\diam A = \sup \{d(x,y) : x,y \in A\}$$
be the diameter of the subset $A$. Following common notation, we use $\inter(A)$ and $\cl{A}$ to denote, respectively, the interior and closure of $A$ (the ambient space for the closure operation will be understood from context). Also, we let $\partial A = \cl{A} \backslash \inter(A)$ denote the boundary of $A$. Lastly, for $\e >0$, let
$$N_{\e}(A) = \{ x \in X : \dist(x,A) < \e \}$$
be the open $\e$-neighborhood of $A$.

Suppose that $(X,d)$ is compact, and let $\mathcal{U} = \{U_i\}_{i \in I}$ be a finite open cover. By the Lebesgue lemma, there is a positive constant $\delta >0$ such that each ball $B(x,2\delta)$ lies entirely in some set $U_i$. Let 
$$f_i(x) = \min \left\{ 1, \tfrac{1}{\delta}\dist(x, N_{\delta}(X \backslash U_i) \right\},$$
which is a $1/\delta$-Lipschitz function with values in $[0,1]$ and whose support is contained in $U_i$. Moreover, for each $x \in X$, we have
$$f(x) := \sum_{i \in I} f_i(x) \geq 1,$$
as $B(x,2\delta) \subset U_i$ for some $i$. If $N = \max \{\sum_i \chi_{U_i}(x) : x\in X \}$, then also $f(x) \leq N$ for all $x$. Define $\phi_i(x) = f_i(x)/f(x)$ so that the following properties hold:
\begin{itemize}
\item[(i)] $\phi_i$ is $(2N+1)/\delta$-Lipschitz with support contained in $U_i$;
 \item[(ii)] $0 \leq \phi_i(x) \leq 1$ for all $x \in X$;
\item[(iii)] $\sum_i \phi_i(x) = 1$ for all $x \in X$.
\end{itemize}
The family $\{\phi_i\}$ is therefore a $(2N+1)/\delta$-Lipschitz \ti{partition of unity} subordinate to $\mathcal{U}$ \cite[Chapter~2]{Hein05}. Partitions of unity are useful in general metric settings to produce a type of proxy for linear structure.

There is a canonical way to associate a simplicial complex to the cover $\mathcal{U}$ whose combinatorics mimics the combinatorics of $\mathcal{U}$. Often, this simplicial complex is defined as an abstract complex that encodes the intersections among the sets in $\mathcal{U}$. We prefer to work with a geometric realization of this complex in Euclidean space. For ease, then, let us index the collection $\{U_i\}$ by the integers $1,\ldots,M$, and let $e_i$ be the $i$-th standard basis vector in $\R^M$.

\begin{definition}
The nerve of $\mathcal{U}$, denoted by $\Ner(\mathcal{U})$, is
$$\Ner(\mathcal{U}) = \bigcup \{ \conv(e_{i_0},\ldots,e_{i_m}) : U_{i_0}\cap \cdots \cap U_{i_m} \neq \emptyset \},$$
where the union runs over collections of sets in $\mathcal{U}$ that have non-empty intersection.
\end{definition}

Here, and in general, we use $\conv(A)$ to denote the convex hull of a set $A \subset \R^M$. When $A = \{a_0,\ldots,a_m\}$ is a finite set, we can express
\begin{equation} \label{convex}
\conv(A) = \left\{ \sum_{i=0}^m \lambda_i a_i : \lambda_i \geq 0 \text{ and } \lambda_0+\ldots+\lambda_m = 1\right\},
\end{equation}
and if $m \leq M$, then this is a (possibly degenerate) $m$-dimensional simplex in $\R^M$. Thus, the simplex spanned by $e_{i_0},\ldots,e_{i_m}$ in $\R^M$ is in the nerve of $\mathcal{U}$ if, and only if, the corresponding sets $U_{i_0},\ldots,U_{i_m}$ have a common intersection.

The partition of unity $\{\phi_i \}_{i \in I}$ allows us to map $X$ naturally to $\Ner(\mathcal{U})$. Namely, define $\phi \colon X \rightarrow \Ner(\mathcal{U})$ by
\begin{equation} \label{phi}
\phi(x) = \sum_{i \in I} \phi_i(x) e_i, \hspace{0.5cm} x \in X,
\end{equation}
and note that $\phi$ is continuous. In fact, as each $\phi_i$ is Lipschitz, the map $\phi$ will be Lipschitz as well. The fact that $\phi(x) \in \Ner(\mathcal{U})$ follows immediately from the definition of the nerve, the characterization in \eqref{convex}, and the properties (ii) and (iii) above.

It will be useful for us later to subdivide the simplices in the nerve without changing $\Ner(\mathcal{U})$ as a set in $\R^M$. The \ti{barycentric subdivision} allows us to do this in a canonical way. Once again, we work with a geometric realization of the relevant complexes. 

Let $S \subset \R^M$ be a simplicial complex whose simplices are convex hulls of the standard basis vectors $e_i$. For each collection $\{e_{i_0},\ldots,e_{i_m}\}$ of vertices which generate a simplex in $S$, we define its \textit{barycenter} to be the point
$$\bc(e_{i_0},\ldots,e_{i_m}) = \tfrac{1}{m+1} \lp e_{i_0}+\ldots+e_{i_m} \rp .$$
The subdivision proceeds inductively, by dimension, on the simplices in $S$. Intuitively, we may think about it in the following way. First, subdivide each edge by adding a vertex at $\bc(e_i,e_j)$ whenever $\conv(e_i,e_j)$ is in $S$. Second, subdivide each 2-dimensional simplex by adding a vertex at $\bc(e_i,e_j,e_k)$ whenever $\conv(e_i,e_j,e_k)$ is in $S$, and then add edges from $\bc(e_i,e_j,e_k)$ to each vertex on the boundary of $\conv(e_i,e_j,e_k)$ (these vertices may come from $S$ itself or from the first step in the subdivision). Continue in the same way, until each simplex in $S$ has been subdivided. For further reference, see \cite[pp.\ 119--120]{Hat02}.

The resulting simplicial complex is called the first barycentric subdivision of $S$. Observe that the geometric realizations of the complexes we obtain throughout this process, including in the final step, coincide with $S$ \textit{as sets in $\R^M$}. We will, however, use $S_b$ to denote the geometric realization of this new complex to emphasize the fact that we have a refined simplicial structure. The following fact will be important in later arguments. If $\conv(e_{i_0},\ldots,e_{i_m})$ is an $m$-dimensional simplex in $S$, then after the barycentric subdivision, it is a union of $m$-dimensional simplices in $S_b$ with geometric form
\begin{equation} \label{simplexeq}
\conv(p_0,\ldots,p_m),
\end{equation}
where $p_j = \bc(e_{\sigma(i_0)},\ldots,e_{\sigma(i_j)})$ for each $0\leq j\leq m$, and $\sigma$ is a permutation of the indices $i_0,\ldots,i_m$. In particular, $p_0=e_{\sigma(i_0)}$.

Before concluding this section, let us record a topological lemma that will be useful later on. Let $P$ be a compact, convex set in $\R^n$. We say that a closed half-space $H \subset \R^n$ \ti{supports} $P$ if $P \cap H$ is non-empty and is contained in $\partial H$.

\begin{lemma} \label{poly}
Let $P \subset \R^n$ be compact and convex, with non-empty interior. Suppose that $f \colon P \rightarrow \R^n$ is continuous, and for each $x \in \partial P$ there is a closed half-space $H$ which supports $P$, with $x, f(x) \in H$. Then $P \subset f(P)$.
\end{lemma}  

\begin{proof}
As $f(P)$ is compact and $P$ is the closure of $\inter(P)$, it suffices to show that $\inter(P) \subset f(P)$. Aiming for a contradiction, let us suppose that there is $y \in \inter(P) \backslash f(P)$. By translation, we may assume that $y=0$.

The compactness and convexity of $P$ imply that for each $x \in \R^n \backslash \{0\}$, the ray from $0$ through $x$ intersects $\partial P$ in exactly one point, which we denote by $p(x)$. We claim that the map $p \colon \R^n \backslash \{0\} \rightarrow \partial P$ is continuous. To verify this, fix $\e>0$ small enough that $\cl{B}(0,\e) \subset \inter(P)$, and let $\pi \colon \R^n \backslash \{0\} \rightarrow \partial B(0,\e)$ be the canonical projection onto the sphere of radius $\e$. It is clear that $\pi$ is continuous. Observe also that $p(x) = p(\pi(x))$ for all $x \in \R^n \backslash \{0\}$, so it suffices to show that $p$, restricted to $\partial B(0,\e)$, is continuous. For this, we note that $p|_{\partial B(0,\e)}$ is the inverse map of $\pi|_{\partial P}$, the latter of which is a continuous bijection from the compact set $\partial P$ to $\partial B(0,\e)$. Consequently, it is a homeomorphism, so its inverse $p|_{\partial B(0,\e)}$ is also continuous.

Consider the map $g \colon P \rightarrow \partial P$ defined by $g(x) = p(-f(x))$. This is continuous by the assumption that $0 \notin f(P)$. We claim that $g$ has no fixed point. Indeed, if $g(x) = x$, then necessarily $x \in \partial P$. By hypothesis, there is a half-space $H$ of $\R^n$ which supports $P$ and has $x,f(x) \in H$. As $0$ lies in the complement of $H$, we know that $-f(x)$ is in the complement of $H$ as well. The point $p(-f(x))$ lies on the segment joining $0$ and $-f(x)$, and so it also fails to lie in $H$. This, however, contradicts the fact that $g(x)=x \in H$, so $g$ can have no fixed points. The existence of such a map $g$ contradicts the Brouwer fixed point theorem: any continuous map from a compact, convex set in $\R^n$ to itself has a fixed point. Thus, we obtain $\inter(P) \subset f(P)$, as desired.
\end{proof}

\section{A topological length-volume inequality for cubes} \label{cube}

Our primary goal in this section is to prove Theorem \ref{LV}; recall that Theorem \ref{singleweight} follows immediately from it. Before starting the proof, we should remark that our methods are heavily informed by the ideas in O. Schramm's paper on square tilings of rectangles \cite{Sch93}. In fact, there are many similarities between our proof of Theorem \ref{LV} and the proof of \cite[Theorem 5.1]{Sch93}.

Let $[0,1]^n$ be the standard Euclidean unit cube of dimension $n\geq 1$. We will use $F_k$ and $F'_k$, for $1\leq k \leq n$, to denote the pairs of opposite codimension-1 faces of the unit cube:
$$F_k = [0,1]^n \cap \pi_k^{-1}(\{0\}) \hspace{0.5cm} \text{and} \hspace{0.5cm} F_k' = [0,1]^n \cap \pi_k^{-1}(\{1\}),$$
where $\pi_k \colon \R^n \rightarrow \R$ is the projection to the $k$-th coordinate axis.

Let $\mathcal{U} = \{U_i\}_{i \in I}$ be an open cover of $[0,1]^n$, and let $w_k \colon I \rightarrow [0,\infty)$ be corresponding weight functions for $1 \leq k \leq n$. Here, and in what follows, an open cover of $[0,1]^n$ will always mean that the sets are open in the relative topology on $[0,1]^n$, unless otherwise explicitly stated. Using the notation from Section \ref{introsec}, we let $d_k = \dist_{w_k}(F_k,F_k')$. To prove Theorem \ref{LV}, we must show that
\begin{equation} \label{LVeq}
\sum_{i \in I} \lp \prod_{k=1}^n w_k(i) \rp \geq d_1 \cdots d_n.
\end{equation}

To obtain this inequality, we first work under an additional technical assumption on the open cover $\mathcal{U}$. Namely, if there exists $i \in I$ for which $U_i \cap F_k \neq \emptyset$ and $U_i \cap F_k' \neq \emptyset$ for some $k$, then we say that $\mathcal{U}$ is \ti{spanning}. We will verify the desired inequality first in the case that $\mathcal{U}$ is \ti{non-spanning}. After doing this, we will treat the general case by modifying slightly the open cover under consideration. Let us state the intermediate result as a separate proposition.

\begin{prop} \label{inter}
The inequality in \eqref{LVeq} holds under the hypothesis that the cover $\mathcal{U}$ is non-spanning.
\end{prop}

\begin{proof}
We may, of course, assume that $d_k > 0$ for each $k$; otherwise, the desired inequality is trivial. We may also assume that the cover $\mathcal{U}$ is finite. Indeed, compactness guarantees that any cover $\mathcal{U}$ of the cube contains a finite sub-cover. Removing the ``redundant" sets from this collection does not increase the left-hand side of the desired inequality and also does not decrease the distances $d_k$. Our proof will now proceed in several steps, which we explicitly indicate.

\textbf{Step 1: Associate a rectangle to each $U_i$.} For $i \in I$, define
$$d_k(i)  = 
\begin{cases}
0 &\text{ if } U_i \cap F_k \neq \emptyset , \\
\dist_{w_k}(F_k,U_i) &\text{ otherwise}
\end{cases}$$
for $1\leq k\leq n$. Of course, we have $d_k(i) \geq 0$ for each $i$ and $k$. Also note that, by finiteness of the cover $\mathcal{U}$, the infimum in the definition of discrete distances can be replaced by a minimum. More importantly, however, is the fact that if $U_i \cap F_k' \neq \emptyset$, then $d_k \leq d_k(i) + w_k(i)$. This follows immediately from the relevant definitions.

Now define
$$R_i = \prod_{k=1}^n [d_k(i), d_k(i) + w_k(i)], $$
which is an $n$-dimensional rectangle with side lengths $w_k(i)$. To simplify notation, we let
$$I_{k}(i) = \pi_k(R_i) = [d_k(i),d_k(i)+w_k(i)],$$
so that $R_i = \prod_k I_k(i)$. 

We will use $R_i$ as a sort of proxy for the set $U_i$. It will therefore be important that the combinatorics of the rectangles $\{R_i\}_{i \in I}$ mimic those of the sets $\{U_i\}_{i \in I}$, in the following sense.

\begin{claim} \label{cubeclaim}
If $U_i \cap U_j \neq \emptyset$, then $R_i \cap R_j \neq \emptyset$.
\end{claim}

\begin{proof}[Proof of claim]
We simply need to show that $I_k(i) \cap I_k(j) \neq \emptyset$ for each $k$. To this end, fix $k$ and without loss of generality, assume that $d_k(i) \leq d_k(j)$. We claim that $d_k(j) \leq d_k(i) + w_k(i)$. Indeed, there is a chain $U_{i_1},\ldots,U_{i_m}$ that connects $F_k$ and $U_i$ of length $d_k(i) = \dist_{w_k}(F_k, U_i)$; in case $U_i \cap F_k \neq \emptyset$, this is simply the empty chain. As $U_i \cap U_j \neq \emptyset$, the augmented chain $U_{i_1},\ldots,U_{i_m}, U_i$ connects $F_k$ and $U_j$. Thus, $d_k(j) \leq d_k(i) + w_k(i)$, which, along with the assumption that $d_k(i) \leq d_k(j)$, immediately gives $I_k(i) \cap I_k(j) \neq \emptyset$.
\end{proof}

We should remark here that the converse need not hold; there are many configurations, in fact, for which two rectangles intersect even though the corresponding open sets are disjoint. Claim \ref{cubeclaim} easily gives the following.

\begin{claim} \label{cubecor}
If $U_{i_0} \cap \cdots \cap U_{i_m} \neq \emptyset$, then $R_{i_0} \cap \cdots \cap R_{i_m} \neq \emptyset$.
\end{claim}

\begin{proof}[Proof of claim]
As before, it suffices to show that $I_k(i_0) \cap \cdots \cap I_k(i_m) \neq \emptyset$ for each $k$. From the previous claim, we know that $R_{i_j} \cap R_{i_{j'}} \neq \emptyset$ for any pair $j,j'$; in particular, $I_k(i_j) \cap I_k(i_{j'}) \neq \emptyset$. Thus, the $m+1$ intervals $I_k(i_0), \ldots, I_k(i_m)$ have pairwise non-empty intersections. This implies that $I_k(i_0), \ldots, I_k(i_m)$ have a point of common intersection: indeed, the maximum among their left endpoints is at most the minimum among their right endpoints. 
\end{proof}

Now that we have established a correspondence between the combinatorics of $\{U_i\}_{i \in I}$ and $\{R_i\}_{i \in I}$, we wish to map the unit cube $[0,1]^n$ continuously into $\bigcup_i R_i$. To ensure that the image lies within the union of the rectangles, it is technically convenient to pass through the nerve of the cover $\mathcal{U}$. Recall from Section \ref{prelim} that by enumerating $\mathcal{U} = \{U_1,\ldots,U_M \}$, we can express
$$\Ner(\mathcal{U}) = \bigcup \{ \conv(e_{i_0},\ldots,e_{i_m}) : U_{i_0}\cap \cdots \cap U_{i_m} \neq \emptyset \}.$$
The associated partition of unity $\{\phi_i \}$ subordinate to $\mathcal{U}$, which we constructed in the previous section, gives the continuous map 
$$\phi \colon [0,1]^n \rightarrow \Ner(\mathcal{U})$$
that was introduced in \eqref{phi}.

\textbf{Step 2: Map $\Ner(\mathcal{U})$ into $\bigcup_i R_i$.} In order to map $\Ner(\mathcal{U})$ into the union of the rectangles $R_i$, we will pass to the first barycentric subdivision of the nerve and then define our map simplicially. For ease, we use $S$ to denote the complex $\Ner(\mathcal{U})$ and, consistent with earlier notation, the complex obtained after the subdivision will be denoted by $S_b$. As sets in $\R^M$, the complexes $S$ and $S_b$ coincide; moreover, each vertex in $S_b$ arises as the barycenter of a simplex in $S$.

To define $\psi \colon S_b \rightarrow \bigcup_i R_i$, let us first determine where it sends the vertices. Fix such a vertex $p$, so that $p = \bc(e_{i_0},\ldots,e_{i_m})$ for some simplex, $\conv(e_{i_0},\ldots,e_{i_m})$, in the nerve $S$. Note that the choice of $e_{i_0},\ldots,e_{i_m}$ is uniquely determined by $p$, up to order. Then, as $U_{i_0} \cap \cdots \cap U_{i_m} \neq \emptyset,$ Claim \ref{cubecor} above guarantees that
$$R_{i_0} \cap \cdots \cap R_{i_m} \neq \emptyset.$$
We want to send $p$ to a point $z_p$ in this intersection, but we must be careful how to choose it. Recall that
\begin{equation} \label{squareint}
\begin{aligned}
R_{i_0} \cap \cdots \cap R_{i_m} &=  \lp \prod_{k=1}^n I_k(i_0) \rp \cap \cdots \cap \lp \prod_{k=1}^n I_k(i_m) \rp \\
&= \prod_{k=1}^n I_k(i_0) \cap \cdots \cap I_k(i_m),
\end{aligned}
\end{equation}
so choosing $z_p$ in $R_{i_0} \cap \cdots \cap R_{i_m}$ amounts to choosing each coordinate $\pi_k(z_p)$ in the interval 
$$[a_k,b_k] := I_k(i_0) \cap \cdots \cap I_k(i_m).$$ 
We do this according to the following rule. If $U_{i_j} \cap F_k \neq \emptyset$ for each $j$, then we choose $\pi_k(z_p) = a_k$; observe that in this case, $a_k=0$. Otherwise, we choose $\pi_k(z_p) = b_k$.

Let $\psi(p) = z_p$ be as above for the vertices $p$ of $S_b$. Extend $\psi$ to be affine on each simplex in $S_b$ so that $\psi \colon S_b \rightarrow \R^n$ is continuous. We claim that the image is contained in $\bigcup_i R_i$. To see this, first observe that we may express $S_b$ as a union of simplices $\Delta$ that are obtained by subdividing a simplex in $S$ of the same dimension. By \eqref{simplexeq} in the previous section, such simplices have geometric form
$$\Delta= \conv(p_0,\ldots,p_m),$$
where $p_j = \bc(e_{i_0},\ldots,e_{i_j})$ and $\conv(e_{i_0},\ldots,e_{i_m})$ is a simplex in $S$. Consequently,
$$\psi(\Delta) = \conv(\psi(p_0),\ldots,\psi(p_m)) = \conv(z_{p_0},\ldots,z_{p_m}).$$
The choice of $z_{p_j}$ guarantees that
$$z_{p_j} \in R_{i_0} \cap \cdots \cap R_{i_j} \subset R_{i_0}$$
for each $j=0,\ldots,m$. As $R_{i_0}$ is convex, we find $\psi(\Delta) \subset R_{i_0}$.

\textbf{Step 3: Map $[0,1]^n$ into $\bigcup_i R_i$.} We now want to compose $\phi \colon [0,1]^n \rightarrow S$ and $\psi \colon S_b \rightarrow \bigcup_i R_i$ to obtain a map from the unit cube into the collection of rectangles. Recall that the complexes $S$ and $S_b$ coincide as sets in $\R^M$, so we can define
$$f = \psi \circ \phi \colon [0,1]^n \rightarrow \bigcup_{i \in I} R_i,$$
which is continuous. Our goal now is to show that the image of $f$ contains the $n$-dimensional rectangle $ R = \prod_k [0,d_k]$. 

The main claim that we must establish toward this end is that
$$f(F_k) \subset \pi_k^{-1}(\{0\}) \hspace{0.3cm} \text{and} \hspace{0.3cm} f(F_k') \subset \pi_k^{-1}([d_k,\infty)).$$
From here, Lemma \ref{poly} almost immediately implies that $R \subset f([0,1]^n)$. To begin, let $x \in F_k \cup F_k'$, and let $U_{i_0},\ldots,U_{i_m}$ be the sets in $\mathcal{U}$ that contain $x$. Then
$$\phi(x) = \sum_{j=0}^m \phi_{i_j}(x)e_{i_j},$$
and $x \in \bigcap_j U_{i_j}$ implies that $\conv(e_{i_0},\ldots,e_{i_m})$ is a simplex in $\Ner(\mathcal{U})$. Also observe that if $x \in F_k$, then $U_{i_j} \cap F_k \neq \emptyset$ for each $j$; similarly, if $x \in F_k'$, then $U_{i_j} \cap F_k' \neq \emptyset$ for each $j$.

As $\phi(x) \in \conv(e_{i_0},\ldots,e_{i_m})$, after the barycentric subdivision, we know that
$$\phi(x) \in \Delta = \conv(p_0,\ldots,p_m),$$
where $p_j = \bc(e_{i_0},\ldots,e_{i_j})$ for each $j$ (without loss of generality, we may re-order the indices so that $\sigma$ is the identity permutation). Consequently, $f(x) = \psi(\phi(x))$ is a convex combination of the points $\psi(p_j) = z_{p_j}$. It therefore suffices to show that $\pi_k(z_{p_j}) = 0$ for each $j$ if $x \in F_k$, and that $\pi_k(z_{p_j}) \geq d_k$ for each $j$ if $x \in F'_k$. 

In the former case, we have $U_{i_j} \cap F_k \neq \emptyset$, so that $d_k(i_j) = 0$ for each $j$. Consequently, we know that $I_k(i_j) = [0,w_k(i_j)]$, so 
$$I_k(i_0) \cap \cdots \cap I_k(i_j) = [0,b_k(j)]$$
for some $b_k(j) \geq 0$. The choice of $z_{p_j}$ then guarantees that $\pi_k(z_{p_j}) = 0$ for each $j$.

In the latter case, we have $U_{i_j} \cap F'_k \neq \emptyset$, so that
$$d_k \leq d_k(i_j) + w_k(i_j)$$
for each $j$. In particular, 
$$I_k(i_0) \cap \cdots \cap I_k(i_j) = [a_k(j),b_k(j)]$$
for some $a_k(j) \geq 0$ and $b_k(j) \geq d_k$. As $\mathcal{U}$ is non-spanning, we know that $U_{i_j} \cap F_k = \emptyset$ for each $j$. By the choice of $z_{p_j}$, we therefore have 
$$\pi_k(z_{p_j}) = b_k(j) \geq d_k$$
for each $j$, as desired.

It is now straightforward to conclude the proof using Lemma \ref{poly}. Namely, let $H_k$ be the half-space $\pi_k^{-1}([0,\infty))$, and let $H_k'$ be the half-space $\pi_k^{-1}((-\infty,d_k])$, so that 
$$R =\prod_{k=1}^n [0,d_k] = \lp \bigcap_{k=1}^n H_k \rp \cap \lp \bigcap_{k=1}^n H_k' \rp.$$
Let $G_k = R \cap \partial H_k$ and $G_k' = R \cap \partial H_k'$ be the faces of $R$ corresponding to $F_k$ and $F_k'$, respectively, and let $g \colon R \rightarrow [0,1]^n$ be the linear map with $g(G_k) = F_k$ and $g(G_k') =F_k'$. We showed above that $f(F_k) \subset \R^n \backslash \inter(H_k)$ and $f(F_k') \subset \R^n \backslash \inter(H_k')$, so the composition $f \circ g \colon R \rightarrow \R^n$ has 
$$f \circ g(G_k) \subset \R^n \backslash \inter(H_k) \hspace{0.3cm} \text{and} \hspace{0.3cm} f \circ g(G_k') \subset \R^n \backslash \inter(H_k').$$
Lemma \ref{poly} then guarantees that $R \subset f \circ g(R) = f([0,1]^n)$. As $f([0,1]^n) \subset \bigcup_i R_i$, volume considerations immediately give
$$d_1\cdots d_n = \Vol(R) \leq \sum_{i\in I} \Vol(R_i) = \sum_{i \in I} \lp \prod_{k=1}^n w_k(i) \rp,$$
as desired.
\end{proof}

It is not difficult now to prove Theorem \ref{LV}; we only need to argue that the non-spanning assumption in Proposition \ref{inter} is not necessary.

\begin{proof}[\textbf{Proof of Theorem \ref{LV}}]
Let $\mathcal{U}$ be an open cover of $[0,1]^n$, let $w_k$ be associated weight functions, and let $d_k = \dist_{w_k}(F_k,F_k')$ be the corresponding distances, as in the statement of the theorem. Just as in the beginning of the proof of Proposition \ref{inter}, it suffices to assume that $\mathcal{U}$ is finite. Our goal is to modify the cover and the weights slightly in order to obtain a new cover to which we can apply Proposition \ref{inter}. We will do this in such a way that the ``volume" and the ``lengths" associated to the new cover are very close to the original quantities. We will perform this modification in multiple steps.

First, we wish to modify $\mathcal{U}$ to obtain an open cover of the cube so that any two sets either intersect or have strictly positive distance from each other. To this end, let $\delta_1 >0$ be small enough that for each $x \in [0,1]^n$, the ball $B(x,\delta_1)$ is entirely contained in some $U_i$. Also, let $\delta_2 >0$ be small enough so that whenever $U_i \cap U_j \neq \emptyset$, there is some point $z$ with $B(z,\delta_2) \subset U_i \cap U_j$. Similarly, let $\delta_3 > 0$ be small enough so that whenever $U_i \cap F_k \neq \emptyset$, there is $z \in F_k$ with $B(z,\delta_3) \subset U_i$, and whenever $U_i \cap F_k' \neq \emptyset$, there is $z \in F_k'$ with $B(z,\delta_3) \subset U_i$. Now define $\delta = \min \{\delta_1, \delta_2, \delta_3 \}$. 

For each $i$, let $A_i = \{ j : U_i \cap U_j = \emptyset \}$, $B_i = \{k : U_i \cap F_k = \emptyset \}$, and $B_i' = \{k : U_i \cap F_k' = \emptyset \}$. We then form the sets
$$\tilde{U}_i = U_i \backslash \lp \bigcup_{j \in A_i} \overline{N}_{\delta/2}(U_j) \cup \bigcup_{k \in B_i}  \overline{N}_{\delta/2}(F_k) \cup \bigcup_{k \in B_i'}  \overline{N}_{\delta/2}(F_k') \rp,$$
where $\overline{N}_{\e}(V)$ denotes the closed $\e$-neighborhood of $V$. Each $\tilde{U}_i$ is open, and as $\delta \leq \delta_1$, it is clear that $\bigcup_i \tilde{U}_i$ contains $[0,1]^n$.

We claim that for each $i$ and $j$, either $\tilde{U}_i \cap \tilde{U}_j \neq \emptyset$ or $\dist(\tilde{U}_i,\tilde{U}_j) \geq \delta/2$. Indeed, if $\dist(\tilde{U}_i,\tilde{U}_j) < \delta/2$, then there are $x \in \tilde{U}_i$ and $y \in \tilde{U}_j$ with $|x-y| < \delta/2$. This implies that $U_i \cap U_j \neq \emptyset$, for if not, then $j \in A_i$ so that $x$ could not be in $\tilde{U}_i$. Choose $z \in U_i \cap U_j$ with $B(z,\delta) \subset U_i \cap U_j$, which is possible because $\delta \leq \delta_2$. Then it must be that $z \in \tilde{U}_i \cap \tilde{U}_j$. Indeed, suppose that $z \notin \tilde{U}_i$. Then either there is $l \in A_i$ with $z \in \cl{N}_{\delta/2}(U_l)$ or there is $k \in B_i \cup B_i'$ with $z \in \cl{N}_{\delta/2}(F_k) \cup \cl{N}_{\delta/2}(F_k')$. However, as $B(z,\delta) \subset U_i$, the distance from $z$ to any of these $U_l$, $F_k$, or $F_k'$ is strictly larger than $\delta/2$. This immediately rules out $z \in \cl{N}_{\delta/2}(U_l)$ or $z \in \cl{N}_{\delta/2}(F_k) \cup \cl{N}_{\delta/2}(F_k')$. The argument for $z \in \tilde{U}_j$ is the same.

Similarly, we can also show that for each $i$ and $k$, either $\tilde{U}_i \cap F_k \neq \emptyset$ or $\dist(\tilde{U}_i,F_k) \geq \delta/2$. Indeed, if $\dist(\tilde{U}_i, F_k) < \delta/2$, then $U_i \cap F_k \neq \emptyset$. This implies that there is $z \in F_k$ for which $B(z,\delta) \subset U_i$. As $z \notin \cl{N}_{\delta/2}(U_j)$ for each $j \in A_i$, we necessarily have $z \in \tilde{U}_i$. Hence, $\tilde{U}_i \cap F_k \neq \emptyset$. The same arguments also show that for each $i$ and $k$, either $\tilde{U}_i \cap F_k' \neq \emptyset$ or $\dist(\tilde{U}_i,F_k') \geq \delta/2$. Thus, the collection $\{\tilde{U}_i\}_{i \in I}$ has the convenient property that for every incidence relevant to the calculation of a combinatorial distance, the associated sets either intersect or are of distance $\geq \delta/2$ from each other.

Let us again modify the collection $\{ \tilde{U}_i \}_{i \in I}$ slightly, in a way dependent on a parameter $\e >0$ that we will eventually send to 0. Namely, let $0<\e<\delta/(8\sqrt{n})$ be very small. For each $i$, let $V_i = N_{\e/2}(\tilde{U}_i)$, where the neighborhood is now taken in $\R^n$. Thus, each $V_i$ is open in $\R^n$, and the union $\bigcup_i V_i$ contains $[0,1]^n$ but does not intersect any of the half-spaces $\pi_k^{-1}([1+\e/2,\infty))$.

To the collection $\{V_i\}_{i \in I}$ we add small Euclidean balls to produce a cover of the cube $[0,1+\e]^n$. Namely, we can find a collection of points $\{x_j\}_{j \in J}$ with the following properties: $\#J \leq C_n(1/\e)^{n-1}$, where $C_n$ is a dimensional constant; each point $x_j$ lies in one of the codimension-1 spaces $\{ x \in \R^n : \pi_k(x)=1+\e \}$; and the balls $B_j = B(x_j,\sqrt{n}\e)$ have
$$(1,1+\e]^n \subset \bigcup_{j \in J} B_j.$$
Let $\mathcal{V}$ denote the collection of open sets $\{V_i\}_{i \in I} \cup \{B_j\}_{j \in J}$ so that $\mathcal{V}$ is an open cover of the cube $[0,1+\e]^n$. Let
$$G_k = \pi_k^{-1}(\{1+\e\}) \cap [0,1+\e]^n$$
be the codimension-1 face of $[0,1+\e]^n$ opposite to $F_k$, and observe that no set in $\mathcal{V}$ intersects both $F_k$ and $G_k$. In other words, $\mathcal{V}$ is non-spanning (the fact that we are covering a slightly larger cube is not a problem; indeed, Proposition \ref{inter} applies equally well to topological cubes). 

To obtain weight functions for $\mathcal{V}$, we of course want to use the original weights $w_k$ associated to the cover $\mathcal{U}$. Namely, let $v_k$ be weight functions for $\mathcal{V}$ defined by
$$v_k(V_i) = w_k(U_i) \hspace{0.3cm} \text{and} \hspace{0.3cm} v_k(B_j) = \e,$$ 
and let $\tilde{d}_k = \dist_{v_k}(F_k,G_k)$ be the associated distance between opposite faces of the cube. We claim that $d_k \leq \tilde{d}_k$ for each $k$. To see this, let us first establish that any chain in $\mathcal{V}$ of minimal $v_k$-length that connects $F_k$ to $G_k$ has the form
$$V_{i_1}, \ldots, V_{i_m}, B_j$$
for some collection $i_1,\ldots,i_m$ and some $j$. It is clear that the chain must end with some $B_j$, as none of the $V_i$ intersect $G_k$. Also, each ball $B_j$ intersects $G_k$, so the penultimate set in the chain cannot be some other ball $B_l$. Lastly, note that if $V_i, B_j, V_{i'}$ appears in the chain, then $\dist(V_i,V_{i'}) \leq \diam B_j \leq 2\sqrt{n}\e$, so that
$$\dist(\tilde{U}_i,\tilde{U}_{i'}) \leq 4 \sqrt{n}\e < \delta/2.$$
Consequently, $\tilde{U}_i \cap \tilde{U}_{i'} \neq \emptyset$, which also means that $V_i \cap V_{i'} \neq \emptyset$. Thus, in a minimal chain, a ball $B_j$ never appears between two of the $V_i$'s.

Let $V_{i_1}, \ldots, V_{i_m}, B_j$ be a chain of minimal $v_k$-length from $F_k$ to $G_k$. As $V_{i_l} \cap V_{i_{l+1}} \neq \emptyset$ for each $l$, we know that $\dist(\tilde{U}_{i_l},\tilde{U}_{i_{l+1}}) \leq \e < \delta/2$. Thus, $\tilde{U}_{i_l} \cap \tilde{U}_{i_{l+1}} \neq \emptyset$, so also $U_{i_l} \cap U_{i_{l+1}} \neq \emptyset$. Hence, $U_{i_1}, \ldots, U_{i_m}$ is a chain in the collection $\mathcal{U}$. Moreover, $\dist(\tilde{U}_{i_1},F_k) \leq \e/2 < \delta/2$, so $\tilde{U}_{i_1} \cap F_k \neq \emptyset$, and also $U_{i_1} \cap F_k \neq \emptyset$. Similarly, 
$$\dist(\tilde{U}_{i_m},F_k') \leq \e/2 + \diam B_j \leq \e/2 + 2\sqrt{n}\e < \delta/2,$$ so that $\tilde{U}_{i_m} \cap F_k' \neq \emptyset$, and also $U_{i_m} \cap F_k' \neq \emptyset$. Therefore, the chain $U_{i_1}, \ldots, U_{i_m}$ connects $F_k$ and $F_k'$, which implies that
$$d_k \leq \sum_{l=1}^m w_k(U_{i_l}) = \sum_{l=1}^m v_k(V_{i_l}) = \tilde{d}_k - \e \leq \tilde{d}_k.$$

Applying Proposition \ref{inter} to the collection $\mathcal{V}$ with weight functions $v_k$ gives
$$\prod_{k=1}^n d_k \leq \prod_{k=1}^n \tilde{d}_k \leq \sum_{i \in I} \lp \prod_{k=1}^n v_k(V_i) \rp + \sum_{j \in J} \e^n \leq \sum_{i \in I} \lp \prod_{k=1}^n w_k(U_i) \rp + C_n \e,$$
where the last inequality follows from the bound $\#J \leq C_n(1/\e)^{n-1}$. As this holds for any $0<\e<\delta/(8\sqrt{n})$, we send $\e$ to zero to obtain 
$$\prod_{k=1}^n d_k \leq \sum_{i\in I} \lp \prod_{k=1}^n w_k(i) \rp.$$
\end{proof}

\begin{remark}
Not surprisingly, the methods we have used to prove Theorem \ref{LV} can be adapted to prove similar inequalities on other convex polyhedra. In the Riemannian setting, Derrick's methods were extended to a much more general framework by M.~Gromov, and this includes a diameter-volume inequality for simplices \cite[Section 7]{Grom83}. We will not attempt to build an analogous framework here, but we should remark that modifying the arguments above to establish discrete diameter-volume inequalities for simplices is fairly straightforward. More specifically, let $\Delta^n$ denote the $n$-dimensional simplex, and let $T_1,\ldots, T_{n+1}$ denote its codimension-1 faces. If $\mathcal{U} = \{U_i\}_{i \in I}$ is an open cover of $\Delta^n$ and $w \colon I \rightarrow [0,\infty)$ is a corresponding weight function, then we define the \ti{diameter} of $\mathcal{U}$, with respect to $w$, to be
$$d_w(\mathcal{U}) = \inf \left\{ \sum_{j=1}^m w(i_j) : 
\begin{array}{l}
\text{the collection } U_{i_1},\ldots, U_{i_m} \text{ contains chains} \\
\text{that connect } T_k \text{ and } T_l \text{ for each pair } k,l
\end{array} \right\}.$$
Adapting the ideas in our proof of Theorem \ref{LV} to this setting, one can prove that 
$$\sum_{i \in I} w(i)^n \geq \frac{d_w(\mathcal{U})^n}{n!}.$$
Moreover if $w\equiv 1$ is constant, then a modification of the arguments used in \cite{Kin14} to prove Theorem \ref{combcube} gives the improved estimate
$$\# \mathcal{U} \geq {n+d_w(\mathcal{U})-1 \choose n}.$$
This is carried out in the author's Ph.D. thesis \cite[Theorem 3.17]{Kinphd}.
\end{remark}

\section{Lower volume bounds in metric spaces} \label{metric}

Using Theorem \ref{LV}, we can prove a similar length-volume inequality for \ti{images} of the Euclidean cube in a metric space. To set this up, let $(X,d)$ be a metric space and let $g \colon [0,1]^n \rightarrow X$ be a continuous map. Let $\mathcal{U} = \{U_i\}_{i \in I}$ be an open cover of $g([0,1]^n)$ with corresponding non-negative weight functions $w_k$ for $1\leq k\leq n$. If $A,B \subset g([0,1]^n)$ are subsets, then we define
$$\dist_{w_k}(A,B) = \inf \left\{ \sum_{j=1}^m w_k(i_j) : 
\begin{array}{l}
U_{i_1},\ldots,U_{i_m} \text{ is a chain} \\
\text{that connects } A \text{ and } B
\end{array} \right\}$$
as before, where \ti{chains} are finite sequences of the $\{U_i\}_{i \in I}$ whose consecutive sets have non-empty intersection. Let
$$d_k(g) = \dist_{w_k}(g(F_k),g(F_k'))$$
be the discrete distance between the images of opposite faces. Of course, this distance depends strongly on $\mathcal{U}$ and $w_k$ as well.

\begin{prop} \label{cor}
For such $g \colon [0,1]^n \rightarrow X$, we have
$$\sum_{i \in I} \lp \prod_{k=1}^n w_k(i) \rp \geq \prod_{k=1}^n d_k(g) .$$
\end{prop}

\begin{proof}
For each $i \in I$, let $V_i= g^{-1}(U_i)$, so that the collection $\mathcal{V} = \{V_i\}_{i \in I}$ forms an open cover of $[0,1]^n$. Let $d_k = \dist_{w_k}(F_k,F_k')$ be the discrete distance associated to the cover $\mathcal{V}$. Observe that if $V_{i_1},\ldots,V_{i_m}$ is a chain in $\mathcal{V}$ that connects $F_k$ and $F_k'$, then $U_{i_1},\ldots,U_{i_m}$ is a chain in $\mathcal{U}$ that connects $g(F_k)$ and $g(F_k')$. Consequently, we have $d_k(g) \leq d_k$ for each $k$. Applying Theorem \ref{LV} to the cover $\mathcal{V}$ with weights $w_k$ gives
$$\prod_{k=1}^n d_k(g) \leq \prod_{k=1}^n d_k \leq \sum_{i \in I} \lp \prod_{k=1}^n w_k(i) \rp,$$
as desired.
\end{proof}

Using this proposition, we can establish a similar inequality relating more standard metric quantities such as Hausdorff measure and metric distance between sets. Recall from Section \ref{introsec} that if $(X,d)$ is a metric space and $E \subset X$ is compact, the $Q$-dimensional Hausdorff content of $E$ is
$$\mathcal{H}_Q^{\infty}(E) = \inf \left\{\sum_{i \in I} (\diam U_i)^Q : \{U_i\}_{i \in I} \text{ is an open cover of } E \right\}.$$
The associated Hausdorff $Q$-dimensional measure is defined to be
$$\mathcal{H}_{Q}(E) = \lim_{\e \searrow 0} \inf \left\{\sum_{i \in I} (\diam U_i)^Q : \{U_i\}_{i \in I} \text{ covers } E \text{ and } \diam U_i < \e \right\},$$
and it is clear that $\mathcal{H}_{Q}(E) \geq \mathcal{H}_Q^{\infty}(E)$. Thus, lower bounds on Hausdorff content are also lower bounds on Hausdorff measure.

\begin{corollary} \label{imagecube}
If $g \colon [0,1]^n \rightarrow X$ is continuous, then 
$$\mathcal{H}_n^{\infty}(g([0,1]^n)) \geq \prod_{k=1}^n \dist(g(F_k),g(F_k')).$$
\end{corollary}

\begin{proof}
Fix an open cover $\{U_i\}_{i \in I}$ of $g([0,1]^n)$, and let $w_k(i) = \diam U_i$ for each $i$ and each $1\leq k \leq n$. Observe that if $U_{i_1},\ldots,U_{i_m}$ is a chain connecting $g(F_k)$ and $g(F_k')$, then
$$\dist(g(F_k),g(F_k')) \leq \sum_{j=1}^m \diam U_{i_j} = \sum_{j=1}^m w_k(i_j).$$
Thus, $\dist(g(F_k),g(F_k')) \leq \dist_{w_k}(g(F_k),g(F_k'))$ for each $k$. By Proposition \ref{cor}, we have
$$\sum_{i \in I} (\diam U_i)^n \geq \prod_{k=1}^n \dist_{w_k}(g(F_k),g(F_k')) \geq \prod_{k=1}^n \dist(g(F_k),g(F_k')).$$
As this holds for any open cover $\{U_i\}_{i \in I}$ of $g([0,1]^n)$, we obtain the desired inequality.
\end{proof}

Let us note that Corollary \ref{imagecube} is sharp in the sense that equality holds when $X=([0,1]^n,\ell_{\infty})$ and $g$ is the identity map. In this case, $\dist_X(F_k,F_k') = 1$ for each $k$, and it is straightforward to show even that $\mathcal{H}_{Q}(X)=1$. Indeed, one can cover $[0,1]^n$ by $2^{jn}$ cubes of side-length $2^{-j}$ for each $j \in \N$, and the $\ell_{\infty}$-diameter of such cubes is $2^{-j}$. Thus, it is important that we \ti{do not} define Hausdorff content (or Hausdorff measure) with the normalizing factor $\Vol(B^n)2^{-n}<1$.

We also observe that in both Proposition \ref{cor} and Corollary \ref{imagecube}, it is not necessary that $X$ be a metric space. The same arguments hold if $X$ is a \ti{pseudometric space}: the distance between distinct points is allowed to be zero. We illustrate this with the following result, which essentially answers a question of Y.~Burago and V.~Zalgaller \cite[p.\ 296]{BZ88}.

\begin{corollary} \label{BurZal}
Let $\rho$ be a pseudometric on $[0,1]^n$, and assume that every open set in the topology determined by $\rho$ is also open in the Euclidean topology. Then
$$\mathcal{H}^{\infty}_{n, \rho}([0,1]^n) \geq \prod_{k=1}^n \dist_{\rho}(F_k,F_k').$$
\end{corollary}

Here, $\mathcal{H}^{\infty}_{n, \rho}([0,1]^n)$ and $\dist_{\rho}(F_k,F_k')$ are defined in the same manner as the usual Hausdorff content and distance, but using the pseudometric $\rho$ instead of an actual metric. Burago and Zalgaller asked whether the inequality in Corollary \ref{BurZal} is true with Hausdorff measure in place of Hausdorff content, though it seems that their definition of Hausdorff measure includes the factor $\Vol(B^n)2^{-n}$. Taken on face-value, this question has a negative answer, as is exhibited by the metric space $([0,1]^n,\ell_{\infty})$. Corollary \ref{BurZal} appears to be the strongest statement that holds in the general metric setting. It does not, however, immediately recover the inequalities of Derrick and Almgren that are discussed in \cite[Section 38.1]{BZ88}.

\begin{proof}
The pseudometric $\rho$ on $X=[0,1]^n$ canonically induces a metric $\tilde{\rho}$ on the quotient space $\tilde{X} = X/\sim$, where $\sim$ is the equivalence relation $x \sim x'$ if $\rho(x,x')=0$. Let $\pi \colon X \rightarrow \tilde{X}$ be the associated projection. It is straightforward to see that
\begin{equation} \label{pseudo}
\dist_{\rho}(F_k,F_k') = \dist_{\tilde{\rho}}(\pi(F_k),\pi(F_k')) \hspace{0.3cm} \text{and} \hspace{0.3cm} \mathcal{H}^{\infty}_{n, \rho}(X) = \mathcal{H}^{\infty}_{n, \tilde{\rho}}(\tilde{X}).
\end{equation}
Moreover, $U \subset \tilde{X}$ is open in the metric topology if and only if $\pi^{-1}(U)$ is open in the topology on $[0,1]^n$ determined by $\rho$. The hypothesis of topological compatibility then ensures that $\pi$, when viewed as a map from the Euclidean cube $[0,1]^n$ to the metric space $\tilde{X}$, is continuous. Corollary \ref{imagecube}, along with \eqref{pseudo}, gives the desired conclusion.
\end{proof}

Corollary \ref{imagecube} points us in the following direction: in what generality can one obtain Euclidean-type lower volume bounds in metric spaces? More precisely, for which metric spaces $(X,d)$, does one have 
\begin{equation} \label{volbound}
\mathcal{H}_{n}^{\infty}(B(x,r)) \gtrsim r^n
\end{equation}
for all metric balls $B(x,r)$ with $0< r \leq \diam X$? An immediate consequence of Corollary \ref{imagecube} is that \eqref{volbound} holds whenever $X$ satisfies the following property.

\begin{definition} \label{fatcube}
A metric space $(X,d)$ is said to \ti{admit fat $n$-cubes} if there is $\lambda \geq 1$ such that, for each $x \in X$ and $0< r \leq \diam X$, there is a continuous map $g \colon [0,1]^n \rightarrow B(x, r)$ with $\dist(g(F_k),g(F_k')) \geq r/\lambda$ for each $k$. 
\end{definition}

Euclidean-type lower volume bounds cannot, of course, hold in complete generality: the existence of ``thin necks" or ``outward cusps" in $X$ would hinder large volume in certain regions. Thus, it makes sense to impose appropriate connectivity conditions when addressing such questions. For $\lambda \geq 1$, we say that $X$ is $\lambda$-linearly locally contractible if for each $x \in X$ and $0< r \leq (\diam X)/\lambda$, the metric ball $B(x,r)$ can be contracted within $B(x,\lambda r)$ to a point. We say that $X$ is linearly locally contractible if it is $\lambda$-linearly locally contractible for some value $\lambda$. In this context, it turns out that there is a fairly general result on lower volume bounds, which relies on the following deep fact proved by S.~Semmes.

\begin{theorem}[Semmes {\cite[Theorem 1.29(a)]{Sem96}}] \label{semmes}
Let $(X,d)$ be a closed manifold of dimension $n \geq 2$ that is $N$-doubling and $\lambda$-linearly locally contractible. Then for each $x \in X$ and each $0<r \leq \diam X$, there is a surjective map $f \colon X \rightarrow \Sp^n$ that is $C/r$-Lipschitz and is constant outside of $B(x,r/2)$. Here, $C$ depends only on $n$, $N$, and $\lambda$.
\end{theorem}

Here, $N$-doubling means that every metric ball of radius $r >0$ can be covered by at most $N$ balls of radius $r/2$. It functions as a finite-dimensionality condition for metric spaces. From Semmes's theorem, we immediately obtain an estimate of the form in \eqref{volbound}. Indeed, as $f$ has non-zero degree, $f(B(x,r)) = \Sp^n$, so the Lipschitz bounds imply that
$$\mathcal{H}_n^{\infty}(B(x,r)) \geq \lp \tfrac{r}{C} \rp^n \mathcal{H}_n^{\infty}(\Sp^n) \gtrsim r^n,$$
where the implicit constant depends only on $n$, $N$, and $\lambda$. 

This ``Semmes approach" to lower volume bounds is, in a sense, dual to our approach, which seeks to map a nice space \ti{into $X$}, rather than map $X$ into some other controlled space. The relative ease of building Lipschitz maps from a general metric space into Euclidean spaces (for example, as we did in Section \ref{prelim}) makes the Semmes method viable. Still, it would be desirable to make our dual argument work. In the next section, we carry this out for metric surfaces. The methods we use also allow us to extend \cite[Proposition 4.1]{Kin14}, which gives conditions under which a metric space is bi-Lipschitz equivalent to a ``snowflake."

\section{Fat squares in metric surfaces} \label{n=2sec}

Our work in the previous sections was originally motivated by quasiconformal uniformization problems for metric spaces. Here, one attempts to find global parameterizations of metric spaces which enjoy some sort of ``analytic regularity." These problems are especially important to the geometric study of boundaries of hyperbolic groups. In many situations, though, the metric spaces contain no rectifiable curves, so standard notions of length are useless in this analysis.

The theory that has developed around quasiconformal uniformization problems is quite deep and is especially rich for surfaces. The following theorem is characteristic of the subject. It gives conditions under which a topological sphere must actually be a \ti{quasi-sphere}, i.e., must be a quasisymmetric image of the Euclidean sphere.

\begin{theorem}[Bonk--Kleiner \cite{BK02qs}] \label{BKqs}
Suppose that $(X,d)$ is homeomorphic to $\Sp^2$, is Ahlfors 2-regular, and is linearly locally contractible. Then $X$ is quasisymmetrically equivalent to $\Sp^2$.
\end{theorem}

A metric space $X$ is said to be Ahlfors $Q$-regular if every metric ball $B(x,r)$ with radius $0 < r \leq \diam X$ satisfies the volume estimate $\mathcal{H}_{Q}(B(x,r)) \approx r^Q$, with a uniform implicit constant. Recall from the previous section that linear local contractibility means that every metric ball can be contracted within a slightly larger ball to a point. As quasisymmetric maps distort relative distances by a controlled amount, they preserve linear local contractibility. Of course, the Euclidean sphere $\Sp^2$ is linearly locally contractible, so every quasi-sphere is as well. Thus, the more restrictive hypothesis in the theorem above is that of Ahlfors 2-regularity. Later in this section, we will extend Theorem \ref{BKqs} outside of the $2$-regular setting.

For now, though, let us discuss connectivity conditions that are related to linear local contractibility. A common theme among them is that they appear in the study of boundaries of hyperbolic groups. We first list them and then discuss how they relate to each other. For a metric space $(X,d)$ and $\lambda \geq 1$, we say that
\begin{enumerate}
\item[\textup{(i)}] $X$ is $\lambda$-$\LLC_1$ if, for each $p \in X$ and $0 < r \leq \diam X$, any two points $x,y \in B(p,r)$ can be joined by a continuum in $B(p,\lambda r)$;
\item[\textup{(ii)}] $X$ is $\lambda$-$\LLC_2$ if, for each $p \in X$ and $0 < r \leq \diam X$, any two points $x,y \in X \backslash B(p,r)$ can be joined by a continuum in $X \backslash B(p, r/\lambda)$;
\item[\textup{(iii)}] $X$ is $\lambda$-annularly linearly connected if it is connected and, for each $p \in X$ and $0 < r \leq \diam X$, any two points $x,y \in A(p,r,2r)$ can be joined by a continuum in $A(p,r/\lambda, 2\lambda r)$.
\end{enumerate}
Here, we use $A(p,r,R) = \cl{B}(p,R) \backslash B(p,r)$ to denote the (closed) metric annulus centered at $p$ with inner radius $r > 0$ and outer radius $R > r$. Recall that a continuum is simply a compact connected set. The ``LLC" acronym in $\LLC_1$ and $\LLC_2$ stands for ``linearly locally connected" and should not be confused with linear local contractibility. For convenience, we will also use the acronym ``ALC" in place of ``annularly locally connected." This third condition is the strongest of the three; it was introduced by J.~Mackay in \cite{Mack10}.

\begin{lemma} \label{ALC}
If $X$ is $\lambda$-$\ALC$, then it is $\lambda'$-$\LLC_1$ and $\lambda'$-$\LLC_2$ for some $\lambda'$, depending only on $\lambda$.
\end{lemma}

\begin{proof}
We first verify the $\LLC_2$ property. Let $p \in X$, let $0 < r \leq \diam X$, and fix $x,y \in X \backslash B(p,r)$. Without loss of generality, suppose that $R = d(p,x) \leq d(p,y)$, and let $n$ be the largest integer for which $2^nR < d(p,y)$. Let $x_0 = x$, and for each $k \leq n$, choose $x_k \in X$ with $d(p,x_k) = 2^kR$. This is possible because $X$ is connected. Finally, let $x_{n+1}= y$. The ALC condition guarantees that for each $0 \leq k \leq n$, there is a continuum 
$$E_k \subset A \lp p,\tfrac{2^kR}{\lambda'}, 2^{k+1}R\lambda' \rp$$
connecting $x_k$ and $x_{k+1}$, as long as $\lambda' > \lambda$. In particular, the continuum $E=E_0 \cup \cdots \cup E_n$ connects $x$ and $y$ and is contained in $X \backslash B(p,r/\lambda')$. Thus, $X$ is $\lambda'$-$\LLC_2$ for every $\lambda' > \lambda$.

To verify the $\LLC_1$ condition, fix $x,y \in B(p,r)$. By the connectivity of $X$, we may choose $q \in X$ for which $d(x,q) = d(x,y)/2$. Then
$$x,y \in A \lp q, \tfrac{d(x,y)}{4}, \tfrac{3d(x,y)}{2} \rp,$$
and using the same technique as in the previous paragraph, it is not difficult to show that $x$ and $y$ can be connected by a continuum
$$E \subset A \lp q, \tfrac{d(x,y)}{4\lambda'}, \tfrac{3\lambda' d(x,y)}{2} \rp,$$
where $\lambda'$ depends only on $\lambda$. In particular, $E \subset B(p, 5\lambda' r)$. Thus, we see that $X$ is $5\lambda'$-$\LLC_1$.
\end{proof}

When $X$ has some topological regularity, there are close relationships between the $\LLC_1$ and $\LLC_2$ conditions, linear local contractibility, and the ALC condition.

\begin{lemma} \label{LLC}
For $(X,d)$ a closed, connected manifold of dimension $n \geq 2$, the following are true.
\begin{enumerate}
\item[\textup{(i)}] If $X$ is $\lambda$-linearly locally contractible, then it is $\lambda'$-$\LLC_1$ and $\lambda'$-$\LLC_2$ for all $\lambda' > \lambda$.
\item[\textup{(ii)}] If $n=2$ and $X$ is both $\LLC_1$ and $\LLC_2$, then it is linearly locally contractible.
\item[\textup{(iii)}] If $X$ is $\lambda$-linearly locally contractible, then it is $\lambda'$-$\ALC$ for all $\lambda' > \lambda$.
\end{enumerate}
\end{lemma}

\begin{proof}
Parts (i) and (ii) are directly from \cite[Lemma 2.5]{BK02qs}. Part (iii) is effectively a modification of the proof of (i), along with excision. Fix $p \in X$ and $0 < r \leq \diam X$, and let $\lambda' > \lambda$. As a first case, suppose that $r > \diam X/\lambda$, so $B(p,2\lambda'r) = X$. Part (i) implies that $X$ is $\lambda'$-$\LLC_2$, so any two points in $A(p,r,2r)$ can be joined by a continuum in $X \backslash B(p,r/\lambda') = A(p,r/\lambda',2\lambda' r)$. Thus, we may assume that $r \leq \diam X/\lambda$.

Let us define 
$$U = \{ x \in X : x \text{ can be joined to } p \text{ by a path in } B(p, 2\lambda'r) \},$$
which is an open and path connected subset of $B(p,2\lambda'r)$. Part (i) implies that $X$ is $\lambda''$-$\LLC_1$ for any $\lambda'' > \lambda$, and this ensures that $\cl{B}(p,2r) \subset U$. It therefore suffices to show that any two points in $U \backslash B(p,r)$ can be joined by a path in $U \backslash B(p,r/\lambda')$. In terms of reduced homology groups, this is equivalent to the homomorphism 
$$i_{\ast} \colon \tilde{H}_0(U \backslash B(p,r)) \rightarrow \tilde{H}_0(U \backslash B(p,r/\lambda')),$$
induced by the inclusion map, being identically zero. Here we use singular homology with coefficients in $\Z_2$ so that there are no issues with orientation.

As $\lambda' > \lambda$, we may choose $0 < r' < r$ such that $\cl{B}(p, r/\lambda') \subset B(p,r'/\lambda)$. Let $K_1= \cl{B}(p,r/\lambda')$ and $K_2 = \cl{B}(p,r')$ so that both $K_1$ and $K_2$ are compact and 
\begin{equation} \label{contain}
B(p,r/\lambda') \subset K_1 \subset K_2 \subset B(p,r) \subset U.
\end{equation}
Note, additionally, that $K_1 \subset B(p,r'/\lambda)$, which implies that $K_1$ can be contracted within $K_2$ to a point. The containments in \eqref{contain} guarantee that $i_{\ast}$ above will be zero if the homomorphism $\tilde{H}_0(U \backslash K_2) \rightarrow \tilde{H}_0(U \backslash K_1)$ induced by inclusion is zero.

The long exact sequence for relative homology \cite[p.\ 117]{Hat02}, together with path connectedness of $U$, gives that
$$\ldots \rightarrow H_1(U,U\backslash K_i) \overset{\partial_i}{\rightarrow} \tilde{H}_0(U\backslash K_i) \rightarrow \tilde{H}_0(U) = 0$$
is exact for $i=1,2$. Thus, $\partial_i$ is surjective. These sequences are natural with respect to inclusions, so triviality of $\tilde{H}_0(U \backslash K_2) \rightarrow \tilde{H}_0(U \backslash K_1)$ would be guaranteed by triviality of the homomorphism $H_1(U,U\backslash K_2) \rightarrow H_1(U,U\backslash K_1)$, also induced by inclusion.

Applying excision \cite[p.\ 119]{Hat02} to the set $X \backslash U$, which is compact and contained in $X\backslash K_i$, we find that the inclusion $(U,U\backslash K_i) \hookrightarrow (X,X\backslash K_i)$ induces an isomorphism $H_1(U,U\backslash K_i) \rightarrow H_1(X,X\backslash K_i)$.
It is not difficult to see that this isomorphism is natural with respect to inclusions (indeed, it is induced by an inclusion!), so the following diagram commutes:
$$\begin{matrix}
H_1(U,U\backslash K_2) & \rightarrow & H_1(X,X\backslash K_2) \\
\downarrow                    &          			&  \downarrow \\
H_1(U,U\backslash K_1) & \rightarrow & H_1(X,X\backslash K_1)
\end{matrix}$$
where the vertical maps come from the inclusions $(U, U\backslash K_2) \hookrightarrow (U,U\backslash K_1)$ and $(X, X\backslash K_2) \hookrightarrow (X,X\backslash K_1)$. Thus, in order to show that the vertical map on the left is the zero homomorphism, it suffices to prove this for the vertical map on the right. Our goal, then, is to show that $H_1(X,X\backslash K_2) \rightarrow H_1(X,X\backslash K_1)$ is trivial.

For this, we use Poincar\'e duality \cite[p.\ 296]{Spa81}. As $K_i$ is compact, there is an isomorphism $H_1(X,X\backslash K_i) \simeq \check{H}^{n-1}(K_i)$, where $\check{H}^{k}$ denotes the $k$-th Cech cohomology group with $\Z_2$ coefficients. Once again, this isomorphism is natural with respect to inclusions, so triviality of $H_1(X,X\backslash K_2) \rightarrow H_1(X,X\backslash K_1)$ is equivalent to triviality of the homomorphism $\check{H}^{n-1}(K_2) \rightarrow \check{H}^{n-1}(K_1)$ induced by the inclusion $K_1 \hookrightarrow K_2$. Recall, though, that $K_1$ is contractible inside $K_2$, so this latter homomorphism is indeed trivial.

We conclude, then, that the original homomorphism $i_{\ast}$ is identically zero, which means that any two points in $U \backslash B(p,r)$ can be joined by a path in $U\backslash B(p,r/\lambda')$. As $\cl{B}(p,2r) \subset U \subset B(p,2\lambda'r)$, this implies that $X$ is $\lambda'$-ALC.
\end{proof}

A major motivation for introducing these types of connectivity conditions is that they appear in the analysis of hyperbolic groups. Namely, if $G$ is a hyperbolic group and $\bdry G$ denotes its boundary at infinity equipped with a visual metric, then under suitable topological hypotheses, $\bdry G$ will satisfy all of these conditions. For example, if $\bdry G$ is non-empty, connected, and has no local cut points (equivalently, $G$ does not split over a finite group or over a virtually cyclic group), then it is ALC \cite[Proof of Corollary 1.2]{Mack10}. Similarly, if $\bdry G$ is a connected manifold, then it is linearly locally contractible (\cite[Theorem 4.4]{KapB01}, along with \cite[Theorem 3.3]{Kle06}).

Recall from the previous section that $(X,d)$ is said to admit fat $n$-cubes if there is $\lambda \geq 1$ such that, for each $x \in X$ and $0< r \leq \diam X$, there is a continuous map $g \colon [0,1]^n \rightarrow B(x, r)$ with $\dist(g(F_k),g(F_k')) \geq r/\lambda$ for each $k$. The following condition is slightly stronger and is related to the connectivity conditions above.

\begin{definition}
We say that $(X,d)$ \ti{admits fat connecting $n$-cubes} if there is $\lambda \geq 1$ such that for any two distinct points $x,y \in X$ with $d(x,y) \leq (\diam X)/\lambda$, there is a continuous map $g \colon [0,1]^n \rightarrow B(x,\lambda d(x,y))$ with $g(F_1) \subset B(x,d(x,y)/4)$ and $g(F_1') \subset B(y,d(x,y)/4)$, and which has $\dist(g(F_k),g(F_k')) \geq d(x,y)/\lambda$ for each $1 \leq k \leq n$.
\end{definition}

When $n=2$, we use the terminology ``fat connecting squares" in place of ``fat connecting $2$-cubes". The following theorem of J. Mackay will be helpful in comparing the ``fat squares" condition to the other connectivity conditions. Recall that $X$ is said to be $N$-doubling if every ball of radius $r>0$ can be covered by at most $N$ balls of radius $r/2$. An arc $\gamma$ in $X$ is called an $\alpha$-quasiarc if for each pair $x,y \in \gamma$, the sub-arc between them has diameter at most $\alpha d(x,y)$. Similarly, a topological circle in $X$ is called an $\alpha$-quasicircle if for each pair of points $x,y$ on it, there is a sub-arc between them with diameter at most $\alpha d(x,y)$.

\begin{theorem}[Mackay {\cite[Theorem 1.4]{Mack13}}] \label{quasicircle}
Suppose that $(X,d)$ is a complete, $N$-doubling, $\lambda$-$\ALC$ metric space. For any $m \in \N$, there is $\alpha = \alpha(N,\lambda,m) \geq 1$ such that any two distinct points $x,y \in X$ can be joined by $m$ different $\alpha$-quasiarcs, such that the concatenation of any two of them is an $\alpha$-quasicircle.
\end{theorem}

The following lemma tells us that linear local contractibility and annular linear connectivity give fat connecting squares.

\begin{lemma}
Let $(X,d)$ be a complete metric space that is $N$-doubling, $\lambda$-linearly locally contractible, and $\lambda'$-ALC. Then $X$ admits fat connecting squares with constant depending only on $N$, $\lambda$, and $\lambda'$.
\end{lemma}

\begin{proof}
Fix $x,y \in X$ distinct, and for notational ease let $r = d(x,y) >0$. By Theorem \ref{quasicircle}, there is $\alpha \geq 1$, depending only on $N$ and $\lambda'$, such that $x$ and $y$ can be joined by two $\alpha$-quasiarcs in $X$ for which the concatenation of the two is an $\alpha$-quasicircle. Let $\mathcal{C}$ denote this quasicircle, and observe that $\mathcal{C} \subset B(x, \alpha r)$.

Let $\gamma_1$ be a sub-arc of $\mathcal{C}$ with $\gamma_1 \subset B(x,r/4)$ and $\diam \gamma_1 \geq r/8$. Similarly, let $\gamma_1'$ be a sub-arc of $\mathcal{C}$ with $\gamma_1' \subset B(y,r/4)$ and $\diam \gamma_1' \geq r/8$. Then, let $\gamma_2$ and $\gamma_2'$ be the sub-arcs of $\mathcal{C}$ that connect $\gamma_1$ to $\gamma_1'$. Observe that by construction, $\dist(\gamma_1,\gamma_1') \geq r/2$, and by the quasicircle condition, 
$$\dist(\gamma_2,\gamma_2') \geq \frac{1}{\alpha} \min \{\diam \gamma_1, \diam \gamma_1' \} \geq \frac{r}{8\alpha}.$$

Suppose that $r = d(x,y) \leq (\diam X)/(\alpha \lambda)$. As $(X,d)$ is $\lambda$-linearly locally contractible, the ball $B(x, \alpha r)$ is contractible inside $B(x, \alpha \lambda r)$. Thus, there is a continuous map
$$H \colon [0,1] \times B(x, \alpha r) \rightarrow B(x,\alpha \lambda r)$$
for which $H_0 = \id$ and $H_1 \equiv \const$; here, we use the standard notation $H_t(z) = H(t,z)$ for $0\leq t \leq 1$ and $z \in B(x, \alpha r)$. We may restrict this homotopy to the quasicircle, and pre-composing with a parameterization of $\mathcal{C}$, we obtain
$$\tilde{H} \colon [0,1] \times \Sp^1 \rightarrow B(x,\alpha \lambda r),$$
where $\tilde{H}_0(\Sp^1) = \mathcal{C}$ and $\tilde{H}_1 \equiv \const$. This defines a continuous map $\tilde{g} \colon \cl{\D} \rightarrow B(x,\alpha \lambda r)$,
where $\D$ denotes the unit disk in $\R^2$, and $\tilde{g}|_{\Sp^1} = \tilde{H}_0 |_{\Sp^1}$ gives a parameterization of $\mathcal{C}$. More precisely, 
$$\tilde{g}(se^{i\theta}) = \tilde{H}(1-s,e^{i\theta})$$
for each $0 \leq s \leq 1$ and $0 \leq \theta \leq 2\pi$. Notice that $\gamma_1$, $\gamma_1'$, $\gamma_2$, and $\gamma_2'$ correspond to sub-arcs of $\Sp^1$ under this parameterization. It is straightforward to see that we may pre-compose $\tilde{g}$ with an appropriate homeomorphism from $[0,1]^2$ to $\cl{\D}$ to obtain a continuous map
$$g \colon [0,1]^2 \rightarrow B(x,\alpha \lambda r),$$
where $g(F_1) = \gamma_1$, $g(F_1') = \gamma_1'$, $g(F_2) = \gamma_2$, and $g(F_2') = \gamma_2'$. Thus, the map $g$ satisfies $g(F_1) \subset B(x,r/4)$ and $g(F_1') \subset B(y,r/4)$, along with the estimate $\dist(g(F_2),g(F_2')) \geq r/(8\alpha)$. This shows that $X$ admits fat connecting squares with constant 
$\max \{8 \alpha, \lambda \alpha \}$.
\end{proof}

Using Lemma \ref{LLC}(iii) and Corollary \ref{imagecube}, we obtain the following result.

\begin{corollary} \label{fatsquares}
Let $(X,d)$ be a closed, connected manifold of dimension $\geq 2$ that is $N$-doubling and $\lambda$-linearly locally contractible. Then $X$ admits fat connecting squares with constant depending only on $N$ and $\lambda$. In particular, $\mathcal{H}_2^{\infty}(B(x,r)) \gtrsim r^2$ for all balls with $0 < r \leq \diam X$, where the implicit constant again depends only on $N$ and $\lambda$.
\end{corollary}

The conclusion in this corollary recovers, for surfaces, the statement about lower volume bounds that we made in the previous section, based off of Semmes's theorem. A natural question is whether our method can be extended to higher dimensions. We leave this as an open problem.

\subsection{An application to quasi-spheres}

Let us return to the quasisymmetric uniformization result of Bonk and Kleiner in Theorem \ref{BKqs}. Earlier, those same authors had proved a uniformization theorem that holds in all dimensions, albeit with a strong dynamical hypothesis on the metric space.

\begin{theorem}[Bonk--Kleiner \cite{BK02}] \label{BK}
Let $(X,d)$ be a compact, Ahlfors $n$-regular metric space with topological dimension $n \geq 1$, and suppose that $X$ admits a uniformly quasi-M\"obius group action that is cocompact on triples. Then $X$ is quasisymmetrically equivalent to $\Sp^n$.
\end{theorem}

Let us not discuss quasi-M\"obius group actions here but instead mention that the hypothesis imposes strong self-similarity properties on $X$. Comparing Theorem \ref{BKqs} to the $n=2$ case of Theorem \ref{BK}, we see that for surfaces, one can essentially drop the assumption of self-similarity, replacing it by the much weaker linear local contractibility condition, and still obtain the same conclusion regarding quasisymmetric uniformization.

One might be tempted to ask whether this phenomenon occurs in all dimensions. Namely, if $(X,d)$ is homeomorphic to $\Sp^n$, is linearly locally contractible, and is Ahlfors $n$-regular, is it necessarily quasisymmetric to the Euclidean sphere? The answer is ``no" for $n \geq 3$, by counterexamples due to S.~Semmes \cite{Sem96g}, even though such spaces have good analytic properties \cite{Sem96}. See the Introduction of \cite{BK02qs} for further discussion.

In \cite{Kin14}, the present author built off of the methods in \cite{BK02} to obtain a statement similar to that in Theorem \ref{BK} for metric spaces $X$ whose topological dimension and Ahlfors-regular dimension do not coincide. More specifically, this result gives conditions under which a metric space is bi-Lipschitz equivalent to the ``snowflake" of a metric space with equivalent topological and Ahlfors-regular dimensions. 

\begin{theorem}[{\cite[Proposition 4.1]{Kin14}}] \label{desnow}
Let $n \geq 2$ and $0<\e < 1$, and suppose that the metric space $(X,d)$ has the following properties:
\begin{enumerate}
\item[\textup{(i)}] $X$ is homeomorphic to $\Sp^n$;
\item[\textup{(ii)}] $X$ admits a conformal elevator;
\item[\textup{(iii)}] every $\delta$-separated set in $X$ has size at most $C\delta^{-n/\e}$;
\item[\textup{(iv)}] every discrete $\delta$-path from $x$ to $y$ has length at least $C^{-1} (d(x,y)/\delta)^{1/\e}$.
\end{enumerate}
Then there is a metric $\rho$ on $X$ satisfying $\rho \approx d^{1/\e}$, where the implicit constant depends only on the constants from the hypotheses.
\end{theorem}

Here, a ``$\delta$-separated set" is simply a set of points for which pairwise distances are at least $\delta$. Also, by a ``discrete $\delta$-path from $x$ to $y$" we mean a chain of points $x=z_0, z_1, \ldots, z_m = y$ in $X$ with $d(z_i,z_{i-1}) \leq \delta$ for each $i$. The length of this chain is $m$, which is one less than the number of points in the chain. The hypothesis of ``admitting a conformal elevator" is a dynamical condition, akin to the assumption that $X$ admits a uniformly quasi-M\"obius group action that is cocompact on triples. For general metric spaces, this is a very strong hypothesis. Thus, in the same spirit that motivated Theorem \ref{BKqs}, it makes sense to ask whether this assumption is necessary.

A close inspection of the proof of Theorem \ref{desnow} in \cite{Kin14} reveals that the first two hypotheses are used \ti{precisely to verify} that $X$ admits fat connecting $n$-cubes. Thus, we may replace (i) and (ii) by this connectivity assumption. Once this is done, it is also not difficult to modify that proof to deal with metric spaces that are possibly unbounded. Putting these observations together gives the following result.

\begin{prop} \label{nogroup}
Let $n \geq 2$ and $0<\e<1$, and suppose that $(X,d)$ is a metric space for which
\begin{enumerate}
\item[\textup{(i')}] $X$ admits fat connecting $n$-cubes with constant $\lambda$,
\item[\textup{(ii')}] every ball of radius $0 < R \leq \diam X$ can be covered by at most $C (R/r)^{n/\e}$ balls of radius $0 < r < R$,
\item[\textup{(iii')}] every discrete $\delta$-path from $x$ to $y$ has length at least $C^{-1} (d(x,y)/\delta)^{1/\e}$.
\end{enumerate}
Then there is a metric $\rho$ on $X$ satisfying $\rho \approx d^{1/\e}$, where the implicit constant depends only on $\e$, $\lambda$, and $C$.
\end{prop}

The condition in (ii') implies that $X$ has Assouad dimension at most $2/\e$. It is not difficult to show that if $X$ is compact and is Ahlfors $2/\e$-regular, then it automatically satisfies this hypothesis. Specializing the previous proposition to metric surfaces and using Corollary \ref{fatsquares}, we obtain the following corollary, which builds off of Theorem \ref{BKqs}.

\begin{corollary} \label{nogroupcor}
Let $(X,d)$ be a closed metric surface that is linearly locally contractible. Suppose that $X$ is Ahlfors $2/\e$-regular and that every discrete $\delta$-path from $x$ to $y$ in $X$ has length at least $C^{-1} (d(x,y)/\delta)^{1/\e}$. Then there is a metric $\rho$ on $X$ for which $\rho \approx d^{1/\e}$. If, in addition, $X$ is homeomorphic to $\Sp^2$, then $X$ is a quasi-sphere.
\end{corollary}

The final statement in this corollary uses the facts that $(X,\rho)$ is Ahlfors 2-regular and that the identity map between $(X,d)$ and $(X,\rho)$ is quasisymmetric.

\end{document}